\documentclass[12pt]{article}
\usepackage[margin =1in]{geometry}
\usepackage{amssymb,amsthm}
\usepackage{amsmath}
\usepackage[bookmarks=false]{hyperref}
\usepackage{cite}
\usepackage{color}

\numberwithin{equation}{section}

\newcommand{\inclu}[0] {\ar@{^{(}->}}

\newcommand{\dist}{{\rm dist}}
\newcommand{\R}{\mathbb{R}}

\newcommand{\EE}{\mathbb{E}}

\newcommand{\RR}{\mathbb{R}}
\newcommand{\cI}{\mathcal{I}}
\newcommand{\cF}{\mathcal{F}}

\newcommand{\proj}{\mathrm{proj}}
\newcommand{\cX}{\mathcal{X}}

\newcommand{\prox}{{\rm prox}}

\newcommand{\argmin}{\operatornamewithlimits{argmin}}

\newtheorem{thm}{Theorem}[section]

\newtheorem{lem}[thm]{Lemma}
\newtheorem{cor}[thm]{Corollary}

\theoremstyle{remark}

\usepackage{mathtools}

\usepackage[boxruled]{algorithm2e}

\begin{document}
	
	\title{Complexity of finding near-stationary points of convex functions stochastically}
	
	
	\author{Damek Davis\thanks{School of Operations Research and Information Engineering, Cornell University,
			Ithaca, NY 14850, USA;
			\texttt{people.orie.cornell.edu/dsd95/}.}
		\and 
		Dmitriy Drusvyatskiy\thanks{Department of Mathematics, U. Washington, 
			Seattle, WA 98195; \texttt{www.math.washington.edu/{\raise.17ex\hbox{$\scriptstyle\sim$}}ddrusv}. Research of Drusvyatskiy was supported by the AFOSR YIP award FA9550-15-1-0237 and by the NSF DMS   1651851 and CCF 1740551 awards.}
	}

	\date{}
	\maketitle
	\begin{abstract}
	In the recent paper~\cite{weak_conv_papr}, it was shown that the stochastic subgradient method applied to a weakly convex problem, drives the gradient of the Moreau envelope to zero at the rate $O(k^{-1/4})$. In this supplementary note, we present a stochastic subgradient method for minimizing a convex function, with the improved rate $\widetilde O(k^{-1/2})$.
	\end{abstract}
		\section{Introduction}
	Efficiency of algorithms for minimizing smooth convex functions is typically judged by the rate at which the function values decrease along the iterate sequence. A different measure of performance, which has received some attention lately, is the magnitude of the gradient. In the short note \cite{nest_optima}, Nesterov
	 showed  that performing two rounds of a fast-gradient method on a slightly regularized problem  yields an $\varepsilon$-stationary point  
	 in $\widetilde{O}(\varepsilon^{-1/2})$ iterations.\footnote{In this section to simplify notation, we only show dependence on the accuracy $\varepsilon$ and suppress all dependence on the initialization and Lipschitz constants.} This rate is in sharp contrast to the blackbox optimal complexity of $O(\varepsilon^{-2})$ in smooth nonconvex optimization \cite{grad_desc_opt}, trivially achieved by gradient descent. An important consequence is that the prevalent intuition -- smooth convex optimization is easier than its nonconvex counterpart -- attains a very precise mathematical justification. In the recent work \cite{makegradsmall_zhu}, Allen-Zhu investigated the complexity of finding $\varepsilon$-stationary points in the setting when only stochastic estimates of the gradient are available. In this context, Nesterov's strategy paired with a stochastic gradient method (SG) only yields an algorithm with complexity $O(\varepsilon^{-2.5})$. Consequently, the author introduced a new technique based on running SG for logarithmically many rounds, which enjoys the near-optimal efficiency $\widetilde{O}(\varepsilon^{-2})$.
	 
	 In this short technical note, we address a similar line of  questions for nonsmooth convex optimization.  Clearly, there is a caveat: in nonsmooth optimization, it is impossible to find points with small subgradients, within a first-order oracle model. Instead, we focus on the gradients of an implicitly defined  smooth approximation of the function, the Moreau envelope. 
	 
	 Throughout, we consider the optimization problem
	 \begin{equation} \label{eqn:problem_class}
	 \min_{x\in\cX}  ~g(x), 
	 \end{equation} 
	where $\cX \subseteq \RR^d$ is a closed convex set with a computable nearest-point map $\proj_{\cX}$,
	 and $g\colon\R^d\to\R$ a Lipschitz convex function. 	Henceforth, we assume that the only access to $g$ is through a stochastic subgradient oracle; see Section~\ref{sec:conc_guarant} for a precise definition. It will be useful to abstract away the constraint set $\cX$ and define $\varphi\colon\R^d\to\R\cup\{+\infty\}$ to be equal to $g$ on $\cX$	 and $+\infty$ off $\cX$. Thus the target problem \eqref{eqn:problem_class} is equivalent to $\min_{x\in\R^d} \varphi(x)$. In this generality, there are no efficient algorithms within the first-order oracle model that can find $\varepsilon$-stationary points, in the sense of $\dist(0;\partial \varphi(x))\leq \varepsilon$. Instead we focus on finding approximately stationary points of the Moreau envelope:
	 $$\varphi_{\lambda}(x)=\min_{y\in\R^d}~ \{\varphi(y)+\tfrac{1}{2\lambda}\|y-x\|^2\}.$$
	 It is well-known that $\varphi_{\lambda}(\cdot)$ is $C^1$-smooth for any $\lambda>0$, with gradient
	\begin{equation}\label{eqn:grad_form}
	\nabla \varphi_{\lambda}(x)=\lambda^{-1}(x-\prox_{\lambda \varphi}(x)),
	\end{equation}
	 where $\prox_{\lambda \varphi}(x)$ is the proximal point
	 	\begin{equation*}
	 	\prox_{\lambda \varphi}(x):=\argmin_{y\in\R^d}\, \left\{\varphi(y)+\tfrac{1}{2\lambda}\|y-x\|^2\right\}.
	 	\end{equation*}
	When $g$ is smooth, the norm of the gradient $\|\nabla \varphi_{\lambda}(x)\|$ is proportional to the norm of the prox-gradient (e.g. \cite{MR2536820}, \cite[Theorem 3.5]{prox_error}), commonly used in convergence analysis of proximal gradient methods  \cite{nest_conv_comp,Ghadimi2016mini}. In the broader nonsmooth setting, the quantity $\|\nabla \varphi_{\lambda}(x)\|$ nonetheless has an appealing interpretation in terms of near-stationarity for the target problem \eqref{eqn:problem_class}. Namely, the definition of the Moreau envelope directly implies that for any $x\in\R^d$, the proximal point $\hat x:=\prox_{\lambda \varphi}(x)$ satisfies
	\begin{equation*}
	\left\{\begin{array}{cl}
	\|\hat{x}-x\|&=  \lambda\|\nabla \varphi_{\lambda}(x)\|,\\ 
	\varphi(\hat x) &\leq \varphi(x),\\
	\dist(0;\partial \varphi(\hat{x}))&\leq \|\nabla \varphi_{\lambda}(x)\|.
	\end{array}\right. 
	\end{equation*}
	Thus a small gradient $\|\nabla \varphi_{\lambda}(x)\|$ implies that $x$ is {\em near} some point $\hat x$ that is {\em nearly stationary} for \eqref{eqn:problem_class}. The recent paper \cite{weak_conv_papr} notes that following Nesterov's strategy of running two rounds of the projected stochastic subgradient method on a quadratically regularized problem, will find a point $x$ satisfying $\EE\|\nabla \varphi_{\lambda}(x)\|\leq \varepsilon$ after at most $O(\varepsilon^{-2.5})$ iterations. This is in sharp contrast to the complexity $O(\varepsilon^{-4})$ for minimizing functions that are only weakly convex --- the main result of  \cite{weak_conv_papr}. Notice the parallel here to the smooth setting. In this short note, we show that the gradual regularization technique of Allen-Zhu \cite{makegradsmall_zhu}, along with averaging of the iterates, improves the complexity to $\widetilde{O}(\varepsilon^{-2})$ in complete analogy to the smooth setting.

\subsection{Convergence Guarantees}\label{sec:conc_guarant}
Let us first make precise the notion of a stochastic subgradient oracle. To this end, we fix a probability space $(\Omega, \cF, P)$ and equip $\RR^d$ with the Borel $\sigma$-algebra.	
We make the following three standard assumptions: 
\begin{enumerate}
	\item[(A1)]\label{it1} It is possible to generate i.i.d.\ realizations $\xi_1,\xi_2, \ldots \sim dP$.
	\item[(A2)]\label{it2} There is an open set $U$ containing $\cX$ and a measurable mapping $G \colon U \times \Omega \rightarrow \RR^d$ satisfying  $\EE_{\xi}[G(x,\xi)]\in \partial g(x)$ for all $x\in U$.
	\item[(A3)]\label{it3} There is a real $L \geq 0$ such that the inequality, $\EE_\xi\left[ \|G(x, \xi)\|^2\right] \leq L^2$, holds for all $x \in \cX$. 
\end{enumerate}
The three assumption (A1), (A2), (A3) are standard in the literature on stochastic subgradient methods. Indeed, assumptions (A1) and (A2) are identical to assumptions (A1) and (A2) in~\cite{doi:10.1137/070704277}, while Assumption (A3) is the same as the assumption listed in~\cite[Equation~(2.5)]{doi:10.1137/070704277}. 

Henceforth, we fix an arbitrary constant $\rho > 0$ and assume that diameter of $\cX$ is bounded by some real $D>0$. It was shown in \cite[Section~2.1]{davis2018stochastic} that the complexity of finding a point $x$ satisfying $\EE\|\nabla \varphi_{1/\rho}(x)\|\leq \varepsilon$ is at most $O(1)\cdot \frac{(L^2+\varepsilon^2)\sqrt{\rho D}}{\varepsilon^{2.5}}$.  
We will see here that this complexity can be improved to $\widetilde{O}\left(\frac{L^2+\rho^2D^2}{\varepsilon^2}\right)$  by adapting the technique of \cite{makegradsmall_zhu}.


The work horse of the strategy is the subgradient method for minimizing strongly convex functions \cite{Rakhlin_subgrad,MR3353214,hazan_subgrad,subgad_easier}. For the sake of concreteness, we summarize in Algorithm~\ref{alg:subgradient} the stochastic subgradient method taken from \cite{subgad_easier}.
\smallskip

\begin{algorithm}[H]
	\KwData{$x_0 \in \cX$, strong convexity constant $\mu>0$ on $\cX$, maximum iterations $T\in \mathbb{N}$, stochastic subgradient oracle $G$.}
	{\bf Step } $t=0,\ldots,T-2$:\\		
	\begin{equation*}\left\{
	\begin{aligned}
	&\textrm{Sample } \xi_t \sim dP\\
	& \textrm{Set } x_{t+1}=\proj_{\cX}\left(x_{t} - \tfrac{2}{\mu(t+1)}\cdot G(x_t, \xi_t)\right)
	\end{aligned}\right\},
	\end{equation*}
	{\bf Return:} $\bar x=\frac{2}{T(T+1)}\sum_{t=0}^{T-1}(t+1)x_t$.		
	\caption{Projected stochastic subgradient method for strongly convex functions $\textrm{PSSM}^{\rm sc}$($x_0, \mu, G,T$)}
	\label{alg:subgradient}
\end{algorithm}
\smallskip

The following is the basic convergence guarantee of Algorithm~\ref{alg:subgradient}, proved in \cite{subgad_easier}.

\begin{thm}\label{thm:conv_str_subgrad}
The point $\bar x$ returned by Algorithm~\ref{alg:subgradient} satisfies the estimate
$$\EE \left[\varphi(\bar x)-\min \varphi\right]\leq \frac{2L^2}{\mu(T+1)}.$$
\end{thm}

For the time being, let us assume that $g$ is $\mu$-strongly convex on $\cX$. Later, we will add a small quadratic to $g$ to ensure this to be the case. The algorithm we consider follows an inner outer construction, proposed in \cite{makegradsmall_zhu}. We will fix the number of inner iterations  $T \in \mathbb{N}$. and the number of outer iterations $\cI\in \mathbb{N}$. We set $\varphi^{(0)}=\varphi$ and for each $i=1,\ldots, \cI$ define the quadratic perturbations
$$ \varphi^{(i+1)}(x):= \varphi^{(i)}(x)+\mu 2^{i-1}\|x-\hat {x}_{i+1}\|^2.$$
 Each center $\hat{x}_{i+1}$ is obtained by running $T$ iterations of the subgradient method Algorithm~\ref{alg:subgradient} on $\varphi^{(i)}$. We record the resulting procedure in Algorithm~\ref{alg:subgradient_gradual}. We  emphasize that this algorithm is identical to the method in \cite{makegradsmall_zhu}, with the only difference being the stochastic subgradient method used in the inner loop.

\smallskip

\begin{algorithm}[H]
	\KwData{Initial point $x_1 \in \cX$, strong convexity constant $\mu>0$, an averaging parameter $\lambda>0$, inner iterations $T\in \mathbb{N}$, outer iterations $\cI\in \mathbb{N}$, stochastic oracle $G(\cdot,\cdot)$.}
	
	Set	$\varphi^{(0)}=\varphi$, $G^{(0)}=G$, $\hat{x}_0=x_0$, $\mu_0=\mu$.

	{\bf Step } $i=0,\ldots,\cI$:\\	
		\quad Set
		$\hat x_{i+1}=\textrm{PSSM}^{\rm sc}$($\hat{x}_i, \sum_{j=0}^i\mu_j, G^{(i)},T$)
		
		\quad $\mu_{i+1}= \mu\cdot 2^{i+1}$
		
		\quad Define the function and the oracle	$$\varphi^{(i+1)}(x):= \varphi^{(i)}(x)+\frac{\mu_{i+1}}{2}\|x-\hat{x}_{i+1}\|^2\quad\textrm{ and }\quad G^{(i+1)}(x,\xi):=G^{(i)}(x,\xi)+\mu_{i+1}(x-\hat{x}_{i+1}).$$

	{\bf Return:} $\bar x=\frac{1}{\lambda+\sum_{i=1}^{\cI} \mu_i}(\lambda{\hat x}_{\cI+1}+\sum_{i=1}^{\cI} \mu_i{\hat x}_i)$.		
	\caption{Gradual regularization for strongly convex problems $\textrm{GR}^{\rm sc}(x_1, \mu,\lambda, T,  \cI, G)$}
	\label{alg:subgradient_gradual}
\end{algorithm}
\smallskip

Henceforth, let $\mu_i$, $\varphi^{(i)}$, and $\hat{x}_i$ be generated by Algorithm~\ref{alg:subgradient_gradual}. Observe that by construction, equality 
$$\varphi^{(i)}(x)=\varphi(x)+\sum_{j=1}^{i} \frac{\mu_i}{2}\|x-{\hat x}_i\|^2,$$
holds for all $i=1,\ldots, \cI$. Consequently,  it will be important  to relate the Moreau envelope of $\varphi^{(i)}$ to that of $\varphi$. This is the content of the following two elementary lemmas.

\begin{lem}[Completing the square]\label{lem:comb_square}
Fix a set of points $z_i\in\R^d$ and real $a_i>0$,	for $i=1,\ldots, \mathcal{I}$. Define the convex quadratic
	$$Q(y)=\sum_{i=1}^{\cI}\frac{a_i}{2}\|y-z_i\|^2.$$
	Then equality holds: 
	$$Q(y)=Q\left(\bar z\right)+\tfrac{\sum^{\cI}_{i=1} a_i}{2}\|y-\bar z\|^2,$$	
	where $\bar z=\frac{1}{\sum_{i=1}^{\cI} a_i}\sum_{i=1}^{\cI} a_i z_i$ is the centroid.
\end{lem}
\begin{proof}
Taking the derivative shows that $Q(\cdot)$ is minimized  at $\bar z$. The result follows.
\end{proof}

\begin{lem}[Moreau envelope of the regularization]\label{lem:reg_more2}
 Consider a function $h\colon\R^d\to\R\cup\{+\infty\}$ 
  and define the quadratic perturbation $$f(x)=h(x)+\sum_{i=1}^{\cI}\frac{a_i}{2}\|x-z_i\|^2,$$
	 for some $z_i\in\R^d$ and $a_i>0$, with $i=1,\ldots,\cI$. 
  Then for any $\lambda>0$, the Moreau envelopes of $h$ and $f$ are related by the expression
	$$\nabla f_{1/\lambda}(x)=\tfrac{\lambda}{\lambda+A}\left(\nabla h_{1/(\lambda+A)}(\bar x)+\sum_{i=1}^{\cI} a_i(x-z_i)\right),$$
	where we define $A:=\sum_{i=1}^{\cI} a_i$ and $\bar x:=\frac{1}{\lambda+A}\left(\lambda x+\sum_{i=1}^{\cI} a_iz_i\right)$ is the centroid.
\end{lem}
\begin{proof}
		By definition of the Moreau envelope, we have
	\begin{equation}\label{eqn:moreau_represe}
	f_{1/\lambda}(x)=\argmin_y\left\{h(y)+\sum_{i=1}^{\cI}\tfrac{a_i}{2}\|y-z_i\|^2+\tfrac{\lambda}{2}\|y-x\|^2\right\}.
	\end{equation}
	We next complete the square in the quadratic term. Namely define the convex quadratic:
	\begin{equation*}
	Q(y):=\tfrac{\lambda}{2}\|y-x\|^2+\sum_{i=1}^{\cI}\tfrac{a_i}{2}\|y-z_i\|^2.
	\end{equation*}
	Lemma~\ref{lem:comb_square} directly yields the representation 
	$Q(y)=Q(\bar x)+\tfrac{\lambda+A}{2}\|y-\bar x\|^2.$
	Combining with \eqref{eqn:moreau_represe}, we deduce
	$$f_{1/\lambda}(x)=h_{1/(\lambda+A)}(\bar x)+Q(\bar x).$$
	Differentiating in $x$ yields the equalities 
	\begin{align*}
	\nabla f_{1/\lambda}(x)&=\tfrac{\lambda}{\lambda+A}\nabla h_{1/(\lambda+A)}(\bar x)+\lambda\left(\tfrac{\lambda}{\lambda+A}-1\right)(\bar x-x)+\tfrac{\lambda}{\lambda+A}\sum_{i=1}^{\cI} a_i(\bar x-z_i)\\
	&=\tfrac{\lambda}{\lambda+A}\nabla h_{1/(\lambda+A)}(\bar x)+\tfrac{\lambda}{\lambda+A}\sum_{i=1}^{\cI} a_i(x-z_i),
	\end{align*}
	as claimed.
	\end{proof}

The following is the key estimate from \cite[Claim 8.3]{makegradsmall_zhu}.

\begin{lem}\label{lem:key_lem}
	Suppose that for each index $i=1,2,\ldots,\cI$, the vectors ${\hat x}_i$ satisfy
	$$\EE[\varphi^{(i-1)}({\hat x}_i)-\min\varphi^{(i-1)} ]\leq \delta_i.$$
	Then the inequality holds:
	$$\displaystyle\EE\left[\sum_{i=1}^{\cI} \mu_i \|x_{\cI}^*-{\hat x}_i\|\right]\leq 4\sum_{i=1}^{\cI} \sqrt{\delta_i \mu_i},$$
	where 
	$x_{\cI}^*$ is the minimizer of $\varphi^{\cI}$.
\end{lem}

Henceforth, set 
$$M_i:=\sum_{j=1}^{i} \mu_j \qquad \textrm{ and }\qquad  M:=M_{\cI}.$$
By convention, we will set $M_{0}=0$.
Combining Lemmas~\ref{lem:reg_more2} and \ref{lem:key_lem}, we arrive at the following basic guarantee of the method.

\begin{cor}\label{cor:mainaz}
		Suppose for $i=1,2,\ldots,\cI+1$, the vectors ${\hat x}_i$ satisfy
		$$\EE[\varphi^{(i-1)}({\hat x}_i)-\min\varphi^{(i-1)} ]\leq \delta_i.$$
	Then the inequality holds:

	$$\EE\|\nabla \varphi_{1/(\lambda+M)}(\bar x)\|\leq \left(\lambda+2 M\right) \sqrt{\frac{2\delta_{\cI+1}}{\mu+M}}+4\sum_{i=1}^{\cI}\sqrt{\delta_i \mu_i},$$
where $\bar x=\frac{1}{\lambda+M}(\lambda \hat x_{\cI+1}+ \sum_{i=1}^{\cI} \mu_i \hat x_i)$.
\end{cor}
\begin{proof}
Fix an arbitrary  point $x$ and set $\bar x=\frac{1}{\lambda+M}(\lambda x+ \sum_{i=1}^{\cI} \hat x_i)$. Then Lemma~\ref{lem:reg_more2}, along with a triangle inequality, directly implies 
\begin{align*}
\|\nabla \varphi_{1/(\lambda+M)}(\bar x)\|&\leq \left(1+\tfrac{M}{\lambda}\right) \|\nabla \varphi^{(\cI)}_{1/\lambda}(x)\|+\sum_{i=1}^{\cI} \mu_i\|x-\hat{x}_i\|\\
&\leq \left(1+\tfrac{M}{\lambda}\right) \|\nabla \varphi_{1/\lambda}^{(\cI)}(x)\|+\sum_{i=1}^{\cI} \mu_i(\|x-x^*_{\cI}\| +\|x^*_{\cI}-\hat{x}_i\|)\\
&\leq  \left(1+\tfrac{M}{\lambda}\right) \|\nabla \varphi_{1/\lambda}^{(\cI)}(x)\|+M\|x-x_{\cI}^*\|+\sum_{i=1}^{\cI} \mu_i \|x^*_{\cI}-\hat{x}_i\|\\
&\leq(\lambda + 2M)\|x-x_{\cI}^*\|+\sum_{i=1}^{\cI} \mu_i \|x^*_{\cI}-\hat{x}_i\|.
\end{align*}
where the last inequality uses that $\nabla \varphi^{(\cI)}_{1/\lambda}$ is $\lambda$-Lipschitz continuous and $\nabla \varphi^{(\cI)}_{1/\lambda}(x_\cI^\ast) = 0$ to deduce that $\|\nabla \varphi_{1/\lambda}^{(\cI)}(x)\| \leq \lambda \|x - x_\cI^\ast\|$.
Using strong convexity of $\varphi^{\cI}$, we deduce
$$\|x-x_{\cI}^*\|^2\leq \tfrac{2}{\mu+M}(\varphi^{(\cI)}(x)-\varphi^{(\cI)}(x_{\cI}^*)).$$
Setting $x=\hat x_{\cI+1}$, taking expectations, and applying Lemma~\ref{lem:key_lem} completes the proof. 
\end{proof}


Let us now determine $\delta_i>0$ by invoking Theorem~\ref{thm:conv_str_subgrad} for each function $\varphi^{(i)}$. Observe
$$\EE_\xi\|G^{(i)}(x,\xi)\|^2\leq 2(L^2+D^2M_i^2).$$
Thus Theorem~\ref{thm:conv_str_subgrad} guarantees the estimates:
\begin{equation}\label{expect:fromconv_subgrad}
\EE[\varphi^{(i-1)}({\hat x}_i)-\min \varphi^{(i-1)}]\leq\frac{4(L^2+D^2M_{i-1}^2)}{(T+1)(\mu+M_{i-1})},
\end{equation}
Hence for $i=1,\ldots, \cI$, we may set $\delta_i$ to be the right-hand side of \eqref{expect:fromconv_subgrad}. Applying Corollary~\ref{cor:mainaz}, we therefore deduce
\begin{equation}\label{eqn:cray_comp_anal_induct}
\begin{aligned}
\EE\|\nabla \varphi_{1/(\lambda+M)}(\bar x)\|&\leq \left(\lambda+2 M\right) \sqrt{\frac{2\delta_{\cI+1}}{\mu+M}}+4\sum_{i=1}^{\cI}\sqrt{\delta_i \mu_i}\\
&\leq\frac{1}{\sqrt{T+1}}\left((\lambda+2M)\sqrt{\frac{8(L^2+D^2M^2)}{(\mu+M)^2}}+4\sum_{i=1}^{\cI}\sqrt{\frac{4(L^2+D^2M_{i-1}^2)}{(\mu+M_{i-1})}\cdot \mu_i}\right).
\end{aligned}
\end{equation}
Clearly we have $\tfrac{\mu_1}{\mu} = 2$, while for all $i>1$, we also obtain
$$\frac{\mu_i}{\mu+M_{i-1}}\leq \frac{\mu_i}{\mu+\mu_{i-1}}=\frac{2^i}{1+2^{i-1}}\leq 2.$$
Hence, continuing \eqref{eqn:cray_comp_anal_induct}, we conclude 
\begin{align*}
\EE\|\nabla \varphi_{1/(\lambda+M)}(\bar x)\|&\leq \frac{1}{\sqrt{T+1}}\left(\sqrt{8}\cdot(\lambda+2M)\sqrt{\left(\tfrac{ L}{M}\right)^2+ D^2}+8\sqrt{2}\cdot |\cI|\cdot\sqrt{L^2+D^2M^2}\right)
\end{align*}
In particular, by setting $\cI=\log_2(1+\frac{\lambda}{2\mu})$, we may ensure $M = \lambda$. For simplicity, we assume the former is an integer.
Thus we have proved the following key result.
\begin{thm}[Convergence on strongly convex functions]\label{thm:conv_strong_conv}
Suppose $g$ is $\mu$-strongly convex on $\cX$ and we set $\cI=\log_2(1+\frac{\lambda}{2\mu})$ for some $\lambda>0$. Then  $\bar x$ returned by Algorithm~\ref{alg:subgradient_gradual} satisfies
\begin{align*}
\EE\|\nabla \varphi_{1/(2\lambda)}(\bar x)\|&\leq \frac{\left(14\sqrt{2}\cdot\log_2(1+\tfrac{\lambda}{2\mu})\right)\cdot\sqrt{L^2+D^2\lambda^2}}{\sqrt{T+1}}
\end{align*}
\end{thm}

When $g$ is not strongly convex, we can simply add a small quadratic to the function and run Algorithm~\ref{alg:subgradient_gradual}. For ease of reference, we record the full procedure in Algorithm~\ref{alg:subgradient_gradual_not_strong_conv}

\smallskip

\begin{algorithm}[H]
    	\KwData{Initial point $x_{\rm c} \in \cX$, regularization parameter $\mu>0$, an averaging parameter $\lambda>0$, inner iterations $T\in \mathbb{N}$, outer iterations $\cI\in \mathbb{N}$, stochastic oracle $G(\cdot,\cdot)$.}
	
	Set	${\widehat\varphi}(x):=\varphi(x)+\frac{\mu}{2}\|x-x_{\rm c}\|^2$, $\widehat G(x, \xi) =G(x, \xi) + \mu(x - x_{\rm c})$, $x_0= x_{\rm c}$.

	Set $\bar x = \textrm{GR}^{\rm sc}(x_{\rm c}, \mu,\lambda/2, T,  \cI, \widehat G)$

	{\bf Return:} $\bar z=\tfrac{\mu}{\mu+\lambda} x_{\rm c}+\tfrac{\lambda}{\mu+\lambda}\bar x$.		
	\caption{Gradual regularization for non strongly convex problems}
	\label{alg:subgradient_gradual_not_strong_conv}
\end{algorithm}
\smallskip

Our main theorem now follows. 
\begin{thm}[Convergence on convex functions after regularization]
Let $\rho > 0$ be a fixed constant, and suppose we are given a target accuracy $\varepsilon\leq 2\rho D$. Set $\mu:=\frac{\varepsilon}{2D}$,  $\lambda:=2\rho-\frac{\varepsilon}{2D}$, and $\cI = \log_2(\frac{3}{4} + \frac{\rho D}{\varepsilon})$. Then for any $T > 0$, Algorithm~\ref{alg:subgradient_gradual_not_strong_conv} returns a point $\bar z$ satisfying:
\begin{align*}
\EE\|\nabla \varphi_{1/(2\rho)}(\bar z)\|&\leq \frac{\left(28\sqrt{2}\cdot\log_2(\tfrac{3}{4}+\tfrac{\rho D}{\varepsilon})\right)\cdot\sqrt{2L^2+3\rho^2D^2}}{\sqrt{T+1}} + \frac{\varepsilon}{2} 
\end{align*}
Setting the right hand side to $\varepsilon$ and solving for $T$, we deduce that it suffices to make 
$$O\left(\frac{\log^3(\frac{\rho D}{\varepsilon})(L^2+\rho^2D^2)}{\varepsilon^2}\right)$$
calls to $\proj_{\cX}$ and to the stochastic subgradient oracle in order to find a point $\bar z\in \cX$ satisfying $\EE\|\nabla \varphi_{1/(2\rho)}(\bar z)\|\leq \varepsilon$.   
\end{thm}
\begin{proof}
Lemma~\ref{lem:reg_more2} guarantees the bound
$$\left\|\nabla\varphi_{1/(\lambda+\mu)}\left(\tfrac{\mu}{\mu+\lambda} x_{\rm c}+\tfrac{\lambda}{\mu+\lambda}\bar x\right)\right\|\leq \tfrac{\lambda+\mu}{\lambda}\|\nabla \widehat{\varphi}_{1/\lambda}(\bar x)\|+\mu D.$$
Applying Theorem~\ref{thm:conv_strong_conv} with $\lambda$ replaced by $\tfrac{1}{2}\lambda$ and $L$ replaced by $2(L^2 + D^2\mu^2)$, we obtain
$$
\EE\left\|\nabla\varphi_{1/(2\rho)}(\bar z)\right\|\leq \tfrac{\lambda+\mu}{\lambda}\tfrac{\left(14\sqrt{2}\cdot\log_2\left(1+\tfrac{\lambda}{4\mu}\right)\right)\cdot\sqrt{2(L^2+ D^2\mu^2) + \tfrac{1}{4}D^2\lambda^2}}{\sqrt{T+1}}+ \frac{\varepsilon}{2}.
$$ 
Some elementary simplifications yield the result. 
\end{proof}

\let\oldbibliography\thebibliography
\renewcommand{\thebibliography}[1]{%
  \oldbibliography{#1}%
  \setlength{\itemsep}{-1pt}%
}

			\bibliographystyle{plain}
	\bibliography{bibliography}

\end{document}